\documentclass[a4paper,10pt]{article}

\usepackage[dvipdfmx,%
 bookmarks=true,%
 bookmarksnumbered=true,%
 colorlinks=true,%
 pdftitle={},%
 pdfauthor={Yokoyama},%
 pdfsubject={},%
pdfkeywords={}]{hyperref}

\makeatletter
\def\blfootnote{\gdef\@thefnmark{}\@footnotetext}
\makeatother

\usepackage{amsthm}
\usepackage{latexsym}
\usepackage{amsmath}
\usepackage{amscd}
\usepackage{amssymb}
\usepackage{authblk}

\usepackage{url}

\newcommand{\la}[0]{\langle}
\newcommand{\ra}[0]{\rangle}
\newcommand{\restrict}[0]{{\upharpoonright}}
\newcommand{\plus}[0]{{\rm +}}
\newcommand{\hyphen}[0]{\mathchar`-}
\newcommand{\defined}[0]{{\downarrow}}

\newcommand{\N}[0]{\mathbb{N}}

\newcommand{\cS}[0]{\mathcal S}

\newcommand{\thmref}[1]{Theorem \ref{#1}}

\newcommand{\propref}[1]{Proposition \ref{#1}}

\newcommand{\claimref}[1]{Claim \ref{#1}}

\newcommand{\condition}[1][s-1]{(F^0_{#1},\dots,F^{k-1}_{#1},I_{#1})}

\newcommand{\negareq}[1][e]{N^i_{#1}}

\newcommand{\posireq}[1][e]{R_{#1}^i}

\newcommand{\Enega}[1][n]{E_{#1}^{i,-}}
\newcommand{\Eposi}[1][n]{E_{#1}^{i,+}}
\newcommand{\enuposi}[3][i]{\widetilde{J_{{#2}^{#1}_{#3}}}}
\newcommand{\stageindex}[3][i]{{#2}^{#1}_{#3}}

\newcommand{\finsub}[0]{\subseteq_{\text{fin}}}

\newcommand{\T}[0]{\mathrm{T}}

\newcommand{\TJ}[0]{\mathrm{TJ}}

\newcommand{\RCAo}{\mathsf{RCA_0}}
\newcommand{\WKLo}[0]{\mathsf{WKL_0}}
\newcommand{\WKLs}[0]{\mathsf{WKL_0^{*}}}
\newcommand{\RCAs}{\mathsf{RCA_0^{*}}}
\newcommand{\ACAo}{\mathsf{ACA}_0}

\newcommand{\wkl}[0]{\mathrm{WKL}}
\newcommand{\RTtt}[0]{\mathrm{RT}^2_2}
\newcommand{\RTt}[1][\empty]{\mathrm{RT}^2_{#1}}

\newcommand{\BII}[0]{{\mathrm{B}\Sigma^0_2}}
\newcommand{\BIII}[0]{\mathrm{B}\Sigma^0_3}

\newcommand{\BN}[1][n]{\mathrm{B}\Sigma^0_{#1}}

\newcommand{\IN}[1][n]{\mathrm{I}\Sigma^0_{#1}}

\newcommand{\INX}[2]{\mathrm{I}\Sigma^{#2}_{#1}}
\newcommand{\In}[2]{\mathrm{I}\Sigma^{#1}_{#2}}
\newcommand{\Bn}[2]{\mathrm{B}\Sigma^{#1}_{#2}}
\newcommand{\EM}{\mathrm{EM}}
\newcommand{\EMinf}{\mathrm{EM}_{<\infty}}

\newcommand{\sEMinf}{\mathrm{sEM}_{<\infty}}
\newcommand{\DAN}[1][n]{\mathrm{D}^2_{#1}}

\newcommand{\COH}{\mathrm{COH}}
\newcommand{\PA}{\mathrm{PA}}

\newcommand{\effcodf}[0]{\mathcal{U}}
\newcommand{\effcods}[0]{\mathcal{W}}

\newcommand{\Dzt}[1][\empty]{\Delta^0_2{#1}}

\newcommand{\sgz}[1][n]{\Sigma^0_{#1}}
\newcommand{\piz}[1][n]{\Pi^0_{#1}}
\newcommand{\sgo}[1][n]{\Sigma^1_{#1}}
\newcommand{\pio}[1][n]{\Pi^1_{#1}}

\usepackage{amsmath,amssymb,amsmath}
\usepackage{amsthm}
\usepackage[abbrev]{amsrefs}
\usepackage{latexsym}
\usepackage{comment}
\usepackage{amsfonts}

\theoremstyle{definition}
\newtheorem{thm}{Theorem}[section]
\newtheorem{prop}[thm]{Proposition}
\newtheorem{lem}[thm]{Lemma}
\newtheorem{cor}[thm]{Corollary}
\newtheorem{dfn}[thm]{Definition}
\newtheorem{ex}[thm]{Example}
\newtheorem{rem}[thm]{Remark}

\newtheorem{claim}[thm]{Claim}

\usepackage{xcolor}

\definecolor{lightred}{rgb}{1,.60,.60}

\begin{document}

\title{Low-like basis theorems for Ramsey's theorem\\ for pairs in first-order arithmetic}

\author[1]{Hiroyuki Ikari}
\author[2]{Keita Yokoyama}

\affil[1]{\small AXA Life insurance Co., Ltd.%
\footnote{This research was conducted independently of the author's affiliation.}
\\  \sf{hiroyuki.ikari.q8@alumni.tohoku.ac.jp}}
\affil[2]{\small Mathematical Institute, Tohoku University\\ \sf{keita.yokoyama.c2@tohoku.ac.jp}}

\date{}
\maketitle

\begin{abstract}
  We construct an $\ll^2$-solution (also known as a weakly low solution) to $\DAN[\empty]$ within $\BN[3]$ and prove the $\ll^2$-basis theorem for $\RTt$ over $\BN[3]$.
  The $\ll^2$-basis theorem is a variant of the low basis theorem, which has recently received focus in the context of the first-order part of Ramsey type theorems.
  For the construction, we use Mathias forcing in an effectively coded $\omega$-model of $\WKLo$ to ensure sufficient computability under the system with weaker induction.
  Using a similar method, we also show the $\ll^2$-basis theorem for $\RTtt$ and $\EMinf$, a version of Erd\H{o}s-Moser principle, within $\IN[2]$.
  These results provide simpler proofs of known results on the $\Pi^1_1$-conservativities of $\RTt, \RTtt$ and $\EMinf$ as corollaries.
\end{abstract}

\section{Introduction}
In this paper, we reconstruct low-like solutions to Ramsey's theorems for pairs within arithmetic theories with weak induction.
To this end, we investigate low-like basis theorems for the decompositions of Ramsey's theorem for pairs, specifically $\COH, \DAN[2], \DAN[\empty]$, and $\sEMinf$, which have been extensively studied in the analysis of the strength of Ramsey's theorem for pairs.

The strength of Ramsey's theorem for pairs and its variants, particularly $\RTtt$ and $\RTt$, has attracted significant attention as these principles lie strictly between $\RCAo$ and $\ACAo$ yet are not equivalent to $\WKLo$.
In particular, their first-order consequences---number-theoretic statements expressible in the language of first-order arithmetic---have been widely investigated.
Let us briefly review some representative results regarding their strength.
Note that all the results cited here employ $\RCAo$ as the base theory.
Based on the computability-theoretic analysis of Ramsey's theorem by Jockusch \cite{Jockusch}, Simpson \cite{sosoa} showed that $\ACAo$ is equivalent to $\mathrm{RT}^n$ (for $n \geq 3$), while $\WKLo$ does not prove $\RTtt$.
Subsequently, Seetapun and Slaman \cite{SeeSla} showed that $\RTt$ does not prove $\ACAo$, and Liu \cite{Liu} later demonstrated that $\RTt$ does not prove $\WKLo$.
A natural question thus arises: how can we best characterize the strength of $\RTtt$ and $\RTt$?
One important perspective is to calibrate the amount of induction or other first-order (arithmetical) statements derivable from them.
From this viewpoint, Hirst \cite{Hirst} showed that $\RTt$ implies $\BIII$.
In a seminal paper \cite{CJS}, Cholak, Jockusch, and Slaman showed that $\RTtt \plus \IN[2]$ is $\Pi^1_1$-conservative over $\IN[2]$ and that $\RTt \plus \IN[3]$ is $\Pi^1_1$-conservative over $\IN[3]$.

In recent years, there has been significant progress in the study of their first-order strength.
A central open question concerning the first-order part of $\RTtt$ is whether $\RTtt \plus \RCAo$ is $\Pi^1_1$-conservative over $\BII$.
Patey and Yokoyama \cite{PatYoko} provided a partial answer to this question by showing that $\RTtt$ is $\Pi^0_3$-conservative over $\BII$, thereby isolating its proof-theoretic strength.
This result was recently improved to $\Pi^0_4$-conservativity \cite{LPY-Pi04}.
Moreover, the $\Pi^{1}_{1}$-conservativity of $\RTtt$ over $\BII \plus \bigcup_{n}\mathrm{WF}(\omega_{n})$ has also been established \cite{LPY-TAMS2025}.
Regarding the strength of $\RTt$, Slaman and Yokoyama \cite{SlYo} established the $\Pi^1_1$-conservativity of $\RTt \plus \WKLo$ over $\BIII$, thus determining the first-order part of $\RTt$.

Conservation proofs over $\BII$ and $\BIII$ typically employ model constructions based on low$_2$-basis theorems via the ``second-jump control'' technique established in \cite{CJS} (e.g., \cite{SlYo,EM}).
On the other hand, in certain cases, a simpler ``first-jump control'' argument is available over $\BII$ (e.g., \cite{CSY-weak, LPY-TAMS2025}), yielding conservation proofs based on weakly-low solution constructions ($\ll^{2}$-basis theorems, in our terminology).
Recent work by Fiori-Carones, Ko\l{}odziejczyk, Wong, and Yokoyama \cite{iso}, based on the isomorphism theorem for $\WKLs$, demonstrates that $\Pi^{1}_{1}$-conservation over $\BII$ or $\BIII$ can be characterized by $\ll^n$-basis theorems.

In this paper, we pursue this line of research.
We emulate the known constructions for weakly-low solutions to the decompositions of $\RTt$ and $\RTtt$, as well as a new construction for $\EMinf$.
While our main results concern the provability of the $\ll^2$-basis theorem (\thmref{thm:headerforbasistheorem}), we also obtain related $\Pi^1_1$-conservativity results as corollaries.
These results unify those of \cite{CJS, SlYo}, and Towsner and Yokoyama \cite{EM}.
Moreover, when combined with techniques of formalized forcing, our results lead to additional results on proof transformations.
The details will appear in a forthcoming joint paper with Ko\l{}odziejczyk.

\subsection*{Structure of this paper}
Section \ref{pre} is devoted to preliminaries, where we introduce the basic setting of subsystems of second-order arithmetic and formalized discussions within them.
This section also introduces the $\ll^2$-relation and countable coded $\omega$-models of $\wkl$, which play important roles in this paper.
\par
Section \ref{constructions} constitutes the core of the paper, where we construct $\ll^2$-solutions to the splits of $\RTtt$ and $\RTt$.
We explicitly present constructions for $\COH, \EMinf$, and $\DAN[\empty]$.%
\footnote{As for $\COH$, the essence of the proof is not original; it is obtained by formalizing well-known constructions for low-like solutions to $\COH$ (e.g., \cite[Theorem 6.44]{slice}). Although already known, we present the construction to lay the groundwork for a unified treatment of the subsequent proofs, including our main theorem.}

\section{Preliminaries}\label{pre}
In this paper, we employ the usual setting for subsystems of the second-order arithmetic. (See \cite{sosoa} for the details.)
The language of second-order arithmetic is 2-sorted and consists of $0,1, \plus, \cdot, =, <$ and $\in$.
{\it Terms} and {\it formulas} are defined accordingly.
The classes of formulas, $\sgz, \piz, \sgo$ and $\pio$ are defined as follows.
$\sgz[0]$ and $\piz[0]$ are the class of all formulas which contain no quantifiers but bounded quantifiers.
$\sgz[n+1]$ is the class of all formulas of the form $\exists x \varphi$ for $\varphi \in \piz[n]$.
$\piz[n+1]$ is the class of all formulas of the form $\forall x \varphi$ for $\varphi \in \sgz[n]$.
$\sgo[0]$ and $\pio[0]$ are the same class defined as $\bigcup_{n \in \omega}\sgz[n]$.
For a fixed set variable $X$, a class $\Sigma^{X}_{n}$ (resp.~$\Pi^{X}_{n}$) is a class of $\sgz[n]$ (resp.~$\piz[n]$) formulas which do not contain free set variables other than $X$.
We say a formula $\varphi$ is {\it arithmetical} when it is an element of $\Sigma^1_0$.
$\sgo[n+1]$ is the class of all formulas of the form $\exists X \varphi$ for $\varphi \in \pio[n]$.
$\pio[n+1]$ is the class of all formulas of the form $\forall X \varphi$ for $\varphi \in \sgo[n]$.

For a class of formulas $\Gamma$, $\mathrm{I}\Gamma$ and $\mathrm{B}\Gamma$, {\it induction} and {\it bounding principles} are the scheme consisting of the following for the formulas $\varphi \in \Gamma$.
\begin{align*}
(\mathrm{I}\varphi): \varphi(0) \land \forall n(\varphi(n) \to \varphi(n+1) ) \to \forall n \varphi(n), \\
(\mathrm{B}\varphi): \forall i < a \exists j \varphi(i,j) \to \exists b (\forall i <a \exists j <b \varphi(i,j)).
\end{align*}
In this paper, we often focus on $\mathrm{I}\Sigma^{X}_{n}$ and $\mathrm{B}\Sigma^{X}_{n}$. They are finitely axiomatizable with the same set variable $X$, so we will consider $\mathrm{I}\Sigma^{X}_{n}$ or $\mathrm{B}\Sigma^{X}_{n}$ is a single formula with exactly one free variable $X$.
Then $\mathrm{I}\Sigma^{0}_{n}$ and $\mathrm{B}\Sigma^{0}_{n}$ are described as $\forall X(\mathrm{I}\Sigma^{X}_{n})$ and $\forall X(\mathrm{B}\Sigma^{X}_{n})$.

The $\Gamma$-\textit{separation} scheme ($\Gamma$-$\mathrm{Sep}$) consists of the following formulas for all $\varphi,\psi \in \Gamma$.
  \[
  \forall n (\varphi(n) \wedge \psi(n)) \to
  \exists X \forall n (( \varphi(n)\to n\in X)\wedge ( \psi(n)\to n\in X)).
  \]
Finally, we regard $\PA^{-}$ as the axioms for discrete ordered semi-rings.

 Based on the above terminology, we define some systems we use in this paper.
\begin{dfn}
 Systems $\IN[n], \BN[n], \RCAo, \WKLo,$ and $\ACAo$ are defined as follows.
   \begin{align*}
&    \IN[n] := \PA^{-} \plus \mathrm{I}\Sigma^{0}_{n}, \quad
    \BN[n] := \PA^{-} \plus \mathrm{B}\Sigma^{0}_{n},\\
&    \RCAo := \IN[1] \plus \Pi^0_1 \hyphen \mathrm{Sep}, \quad
    \WKLo := \IN[1] \plus \Sigma^0_1 \hyphen \mathrm{Sep}, \quad
    \ACAo := \IN[1] \plus \Pi^0_2 \hyphen \mathrm{Sep}.
   \end{align*}

\end{dfn}
Note that $\Pi^0_1 \hyphen \mathrm{Sep}$ is equivalent to the recursive comprehension axiom $\Delta^{0}_{1}$-$\mathrm{CA}$, $\Sigma^{0}_{1}$-$\mathrm{Sep}$ is equivalent to the weak K\H{o}nig's lemma ($\wkl$, the statement that any infinite $0 \hyphen 1$ tree has an infinite path), and $\Pi^{0}_{2}$-$\mathrm{Sep}$ is equivalent to the arithmetical comprehension axiom ($\mathrm{ACA}$).

\begin{rem}
In this paper, we will mostly work within $\RCAo+\IN$ to construct low-like solutions to Ramsey-type statements (\thmref{thm:headerforbasistheorem}).
However, our constructions usually do not require the full strength of the comprehension scheme of $\RCAo$ since they are first-order statements.
Indeed, it is known that $\RCAo+\IN$ is $\Pi^{1}_{1}$-conservative extension of $\IN$, so we do not need to care the difference of those systems and may write
``provable within $\IN$'' rather than ``provable within $\RCAo+\IN$''.
\end{rem}

\subsection{Formalized computability}

We first define formalized versions of some concepts on computability theory.
Since we are mainly interested in discussions within formal systems, notions in computability should be understood as the formalized versions defined in this subsection, unless otherwise specified.

Based on Kleene's normal form theorem, we fix a universal $\Pi^0_1$ formula $\pi (e, x, A)$  with only the displayed variables, which is of the form $\forall m \pi_{0} (m, e, x, A \restrict m)$ where $\pi_{0}$ is a fixed $\Delta^0_0$ formula.

\begin{dfn}\label{dfn:reduce}
  The following definitions are made within $\RCAo$.
    For $X$ and $Y$, $X \leq_{\T} Y$ is defined as
    \begin{align*}
      \exists e_1, e_2 \forall n_{0} \exists m_{0}\forall n\le n_{0}
      & [ ( n \in X \leftrightarrow \exists m\le m_{0} \neg\pi_{0}(m, e_1, n, Y \restrict m))\\
      &\,\land ( n \notin X \leftrightarrow \exists m\le m_{0} \neg\pi_{0}(m, e_2, n, Y \restrict m)) ].
    \end{align*}
    We call these $e_1$ and $e_2$ (or their pair $\la e_1,e_2 \ra$) {\it $Y$-recursive indices (index)} of $X$.
    We write $X =_{\T} Y$ if both $X \leq_{\T} Y$ and $Y \leq_{\T} X$ hold.
We extend this definition to arithmetically definable sets $X,Y$ as long as they are \textit{regular}, in other words, their initial segments $X\restrict m, Y\restrict m$ are always available as (coded) finite sets. Note that a definable set $X$ is regular if $\Bn{X}{1}$ holds.
\end{dfn}

\begin{dfn}[Turing jump]
The Turing jump of $X$ $X'=\mathrm{TJ}(X)$ is defined as the definable set $\{ \la e, x \ra \mid \pi(e, x, X) \}$.
We also define the {\it $n$-th jump} $X^{(n)}=\mathrm{TJ}(n,X)$ for each $n\in\omega$, by $\mathrm{TJ}(0,X) = X, \mathrm{TJ}(n+1, X) = \TJ( \TJ(n, X) )$.
\end{dfn}

Within $\RCAo$, we may deal with Turing jumps as external (definable) objects. We typically need attention when we apply induction for them (see \thmref{thm:rel-induction}).
Nonetheless, the notion of (definable) low sets still play as the key concept.
\begin{dfn}[Low sets]
  A set $X$ is {\it low relative to $A$} if $X' \leq_{\T} (X \oplus A)'$ holds.
  A set $X$ is {\it low$_n$ relative to $A$} if $X^{(n)} \leq_{\T} (X \oplus A)^{(n)}$ holds.
\end{dfn}
By formalizing the universality of the Turing jump, one can easily see the following.
\begin{thm}\label{thm:rel-induction}
Let $n,m\ge 1$. Then $\RCAo$ proves the following.
\begin{enumerate}
\item $\In{X}{n}$ and $Y \leq_{\T} X$ implies $\In{Y}{n}$.
\item $\Bn{X}{n}$ and $Y \leq_{\T} X$ implies $\Bn{Y}{n}$.
\item $\In{X}{n+m}$ is equivalent to $\In{X^{(m)}}{n}$.
\item $\Bn{X}{n+m}$ is equivalent to $\Bn{X^{(m)}}{n}$.
\item If $Y$ is low$_{m}$ relative to $X$, then $\In{X}{n+m}$ implies $\Bn{X\oplus Y}{n+m}$.
\item If $Y$ is low$_{m}$ relative to $X$, then $\Bn{X}{n+m}$ implies $\Bn{X\oplus Y}{n+m}$.
\end{enumerate}
\end{thm}

The following version of the low basis theorem is also essential throughout this paper.
\begin{thm}[{Low basis theorem \cite[Theorem 3.8]{HaPu}}]\label{thm:lowbasis}
There exists a $\mathsf{bool}(\Sigma^{X}_{1})$ formula $\theta(x,e_{1},e_{2},X)$ such that $\IN[1]$ proves the following. (Here $\mathsf{bool}(\Gamma)$ means the class of formulas obtained by boolean combinations and bounded quantifications of $\Gamma$-formulas.)
For any set $A$ and infinite $0 \hyphen 1$ tree $T \leq_{\T} A$ with the recursive indices $e_{1},e_{2}$, the definable set $W=\{x\mid \theta(x,e_{1},e_{2},A)\}$ is a path through $T$, and $\INX{1}{A \oplus W}$ still holds. Moreover, $\Bn{A}{2}$ implies that $W$ is low relative to $A$.
\end{thm}

We next define the notion of primitive recursion with a possibly external object $Z$.
\begin{thm}[{Primitive recursion \cite[Theorem II. 3]{sosoa}}]\label{thm:primitiverec}
The following is provable within $\RCAo$. Let $Z$ be a definable set which satisfies $\Bn{Z}{1}$ and consider $\mathcal F_Z = \{ f : \mathbb N^{<\mathbb N} \to \mathbb N \mid f \leq_{\T} Z \}$.
Then $\INX{1}{Z}$ implies the following.
  \begin{enumerate}
    \item For $f(y_0, \dots, y_{m-1}, \vec x ) , g_0( \vec x ), \dots, g_{m-1}( \vec x ) \in \mathcal{F} _Z$, the function $h$ defined below exists as an element of $\mathcal{F}_Z$.
    \[
    h(\vec x ) = f ( g_0( \vec x), \dots, g_{m-1}(\vec x), \vec x ).
    \]
    \item For $f(\vec x, y) \in \mathcal{F}_Z$ such that $\forall \vec x \exists y f(\vec x, y) = 0$, the function $h$ defined below exists as an element of $\mathcal{F}_Z$.
    \[
    h(\vec x) = \text{the least } y \text{ such that } f(\vec x, y )=0.
    \]
    \item For $f(z, \vec x, y), g(\vec x) \in \mathcal{F}_Z$, the function $h$ defined below exists as an element of $\mathcal{F}_Z$.
    \[
        h(0, \vec x) = g(\vec x),
        h(y+1, \vec x) = f(h(y, \vec x), \vec x, y).
    \]
  \end{enumerate}
\end{thm}

\begin{rem}
  This theorem guarantees that we can realize any $Z$-primitive recursive procedure within formal systems when $\INX{1}{Z}$ holds.
  Indeed, this fact will be an important observation in the proofs in Section \ref{constructions}.
\end{rem}

The following version of the limit lemma is also the basic tool of our constructions.
\begin{thm}\label{thm:limit-lemma}
The following is provable within $\RCAo$. Let $Z$ be a set such that $\Bn{Z}{2}$ holds.
For a definable set $W$, the following are equivalent.
\begin{enumerate}
 \item $W\le_{\T} Z'$.
 \item There exists a function $f:\N^{2}\to 2$ such that $f\le_{\T}Z$ and for any $n\in\N$, $n\in W\leftrightarrow \lim_{s\to \infty}f(n,s)=1$ and $n\notin W\leftrightarrow \lim_{s\to \infty}f(n,s)=0$.
\end{enumerate}
A (definable) set $W$ is said to be $\Delta^{Z}_{2}$-definable, or simply $\Delta^{Z}_{2}$, if one of the above holds.
\end{thm}

\begin{rem}
Assuming $\Bn{Z}{n}$, we may consider the limit lemma for a (definable) set $W\le_{\T}Z^{(n-1)}$, and thus $\Delta^{Z}_{n}$-definable sets can be defined accordingly.
\end{rem}
\begin{rem}
The equivalence in \thmref{thm:limit-lemma} even holds just with $\In{Z}{1}$ if we consider a weaker notion of the Turing reduction (without assuming the regularity of $W$). Thus, we may consider the class of $\Delta^{0}_{2}$-definable sets even within $\RCAo$.
\end{rem}

\subsection{Ramsey's theorem}
We now consider Ramsey's theorem for pairs and its decompositions by means of cohesiveness and $\Dzt$ $k$-partition.
In the light of the limit lemma, we define a $\Dzt$ $k$-partition $\sqcup_{i<k}A_i = \mathbb N$ by a computable function $f:\N^{2}\to k$ such that $\lim_{s\to \infty}f(n,s)$ exists for all $n\in\N$ with $n\in A_{i}\leftrightarrow \lim_{s\to\infty}f(n,s)=i$.
\begin{dfn}
  \begin{enumerate}
    \item $\mathrm{RT}^n_k$ is the following statement. {\it For any $k$-coloring on $n$-tuples of natural numbers, there exists an infinite homogeneous set for at least one color.}
    \item $\mathrm{RT}^n$ is the statement $\forall k \mathrm{RT}^n_k$.
    \item $\COH$ is the following statement. {\it For any countable sequence of sets $\{ R_i \}_{i \in \mathbb N}$, there exists an infinite set $C$ such that $C \subseteq^{\ast} R_i$ or $C \subseteq^{\ast} R_i^c$ for all $i$.} Here, $X \subseteq^{\ast} Y$ means $X \setminus Y$ is finite.
    \item $\DAN[k]$ is the following statement. {\it For any $\Dzt$ $k$-partition $\sqcup_{i<k}A_i = \mathbb N$, there exists an infinite set $B$ such that $B \subseteq A_i$ for some $i$.}
    \item $\DAN[\empty]$ is the statement $\forall k \DAN[k]$.
  \end{enumerate}
\end{dfn}
  \begin{rem}
    It is clear that $\RTtt$ (resp.~$\DAN[2]$) implies $\RTt$ (resp.~$\DAN$) in standard mathematics by induction on the number of colors.
However, we need to distinguish these two concepts within $\RCAo$ because of the lack of induction.
In fact, the equivalence is even not provable from $\In{1}{0}$, while most of our construction should be done within $\IN[2]$ or $\BIII$.
  \end{rem}
\begin{prop}[{\cite[Lemmas 7.11 and 7.13]{CJS}}]\label{prop:decomposition}
  $\RCAo$ proves the following.
  \begin{enumerate}
    \item $\RTt[k] \leftrightarrow \COH \plus \DAN[k]$,
    \item $\RTt \leftrightarrow \COH \plus \DAN[\empty]$.
  \end{enumerate}
\end{prop}

We also examine another principle, the {\it Erd\H{o}s-Moser principle}.
The Erd\H{o}s-Moser principle ($\EM$) has received attention as a part of the decomposition of Ramsey's theorem for pairs.
Towsner and Yokoyama \cite{EM} considered multi-colored variants of $\EM$ and showed one of them, $\EMinf$, is $\Pi^1_1$-conservative over $\IN[2]$.
We prove $\ll^{2}$-basis theorem for $\EMinf$ within $\IN[2]$ using similar technique to $\COH$ and $\DAN[2]$.
This provides a new simpler proof of the above conservation result.
For some recent observations concerning the Erd\H{o}s-Moser principle, see also \cite{QuPaYo-wf} and \cite{LuMi}.
We consider the Erd\H{o}s-Moser principle concerning fallowness defined below.

\begin{dfn}
  A $k$-coloring $c : [\mathbb N]^2 \to k$ is {\it fallow on} $A \subseteq \mathbb N$ when $c(x,z) \in \{ c(x,y), c(y,z)\}$ for any $x,y,z \in A$ such that $x<y<z$.
\end{dfn}

\begin{dfn}
  $\EMinf$ is the following statement: for any $k$ and $k$-coloring $c$, there exists an infinite set $A \subseteq \mathbb N$ on which $c$ is fallow.
  The principle $\sEMinf$ is the following statement: For any $k$ and stable $k$-coloring $c$, there exists an infinite set $A \subseteq \mathbb N$ on which $c$ is fallow.
  Here a $k$-coloring $c$ is {\it stable} when $\lim_y c(x,y)$ exists for each $x$.
\end{dfn}

\subsection{The $\ll^n$-relation and low-like basis theorems}\label{llrelations}

In this subsection, we introduce the $\ll^n$-relation and effectively coded $\omega$-models of $\WKLo$.
The $\ll^n$-relation was introduced in \cite{iso} to combine recursion theoretic discussions with conservation results.

\begin{dfn}\label{dfn:ll}
  $X \ll^n_A Y$ is defined as {\it any infinite $\Delta^{X \oplus A}_{n}$-definable $0 \hyphen 1$ tree has a $\Delta^{Y \oplus A}_{n}$-definable path}.
\end{dfn}
We omit the subscript $A$ when $A = \emptyset$.
Note that $X \ll^1_{\emptyset} Y$ is equivalent to the statement that $X$ has a PA-degree relative to $Y$.
Most of the properties for $X \ll^n_A Y$ appearing in this subsection are reformulation of properties of PA-degrees.

The following definition can be made within $\RCAo \plus \BN$.
\begin{dfn}\label{dfn:basis}
  Let $P$ be a $\Pi^{1}_{2}$-statement which is of the form $\forall X\exists Y\theta(X,Y)$ where $\theta$ is a $\Sigma^{1}_{0}$ formula.
 The $\ll^n_A$-basis theorem for $P$ is the following statement: for sets $X$, $Z$ and $A$ satisfying $X \ll^n_A Z$, there exists a $\Delta^{Z \oplus A}_{n}$-definable set $\widetilde{Y}$ such that $\theta(X,\widetilde{Y})$ and $(X \oplus \widetilde{Y}) \ll^n_A Z$.
 (Here, $\tilde Y$ is said to be a $P$-solution to $X$.)
\end{dfn}

To construct a solution to the $\ll^{2}$-basis theorem, we use coded $\omega$-models of $\WKLo$ with additional effectivity.
Indeed, the notion of coded $\omega$-models has a strong and important connection with the relation $\ll^n$.
The following theorem is a careful reformulation of the well-known fact that $\WKLo$ proves the existence of a countable coded $\omega$-model of $\WKLo$.
\begin{thm}\label{thm:coded}
Let $(M,\cS)$ be a model of $\RCAo$ and $A \in \cS$. Then there exists a $0 \hyphen 1$ tree $T\in \cS$ such that $T\le_{\T} A$ and satisfies the following condition:
\begin{itemize}
 \item[] if $\mathcal W$ is a definable path of $T$ such that $\In{A\oplus \mathcal W}{1}$ holds, then it is a triple of the form
$\mathcal W = (\la W_i \ra_{i\in M}, f_{\mathcal W}, g_{\mathcal W})$ such that
  \begin{enumerate}
    \item[1.] $A = W_0$,
    \item[2.] $\forall i,j ( W_{f_{\mathcal W}(i,j)} = W_i \oplus W_j)$,
    \item[3.] $\forall e,i ( \forall n \exists \sigma \pi_{0}( n, e, \sigma, W_i\restrict n) \to \forall n \pi_{0}( n, e, W_{g_{\mathcal W}(e,i)} \restrict n, W_i\restrict n))$.\\
    (Recall that $\forall n\pi_{0}( n, e, \sigma, W_i\restrict n)$ is a universal $\Pi^0_1$ formula with displayed variables.)
  \end{enumerate}
\end{itemize}
We call such $\mathcal W$ an effectively coded $\omega$-model of $\WKLo$ containing $A$.
\end{thm}
By combining this with the formalized low basis theorem (\thmref{thm:lowbasis}), we have the following.
\begin{thm}\label{thm:coded2}
If $(M,\cS)$ is a model of $\RCAo$ and $A \in \cS$,
then there exists a $\mathsf{bool}(\Sigma^{A}_{1})$-definable set $\mathcal W$ such that $\In{A\oplus \mathcal W}{1}$ and $\mathcal W = (\la W_i \ra_{i\in M}, f_{\mathcal W}, g_{\mathcal W})$ is an effectively coded $\omega$-model of $\WKLo$ containing $A$.
Moreover, $\Bn{A}{2}$ implies that such $\mathcal W$ can be taken as a low set relative to $A$.
\end{thm}
Note that $(M,\la W_i \ra_{i\in M})$ is indeed a model of $\WKLo$ which consists only of $\mathsf{bool}(\Sigma^{A}_{1})$-definable sets, and they are also low relative to $A$ if $\Bn{A}{2}$ holds in $(M,\cS)$.

\begin{proof}[\it Proof of \thmref{thm:coded}.]
We outline the key ideas for the existence of effectively coded $\omega$-models, since the proof is standard.
The core idea is the same as that in \cite{iso} (Lemma 3.2).
We define a tree $T$ whose path codes an effectively coded $\omega$-model in $(M,\mathcal S)$.
A node $\sigma \in T$ encodes information about a segment of $W_i$.
Roughly speaking, $\sigma$ is defined as follows.
Here $s$ is a certain bound associated with the length of $\sigma$.

\begin{itemize}
  \item $\sigma (\la 0, k \ra)=1$ iff $k \in A\restrict s$,
  \item $\sigma (\la \la 0,i,j\ra, k \ra) =1$ iff $k \in W_i \oplus W_j$,
  \item If $\la 1,e,i \ra <s$ and $\exists \tau \in M^s \forall n<s \pi_{0}(n,e,\tau\restrict n,W_i\restrict n)$, then\\
  $\forall n<s \pi_{0} (n, e, \{x<s \mid \sigma( \la \la 1,e,i \ra,x \ra)=1\}, \{ x<s \mid \sigma(\la i,x \ra)=1\} )$.
\end{itemize}
 Note that $\{ x<s \mid \sigma(\la k,x \ra)=1\}$ works as if it were $W_k \restrict s$.
 If we have the additional assumption $A' \in \mathcal S$, we can apply low basis theorem (in $\mathcal W$) to pick a path low relative to $A$.
In this setting $\la 0, i, j \ra$ and $\la 1, e, i \ra$ work as $f_{\mathcal W}$ and $g_{\mathcal W}$ respectively.
\end{proof}

Within effectively coded $\omega$-models, we may obtain some operators we use in the proof in Section \ref{constructions}.
In other words, one may {\it realize an operator} as a function which is Turing reducible to $\mathcal W$ such that, for given inputs $i_0, \dots, i_{k-1}$, outputs an index of a set in $\mathcal W$ which is obtained from $W_{i_0}, \dots, W_{i_{k-1}}$.
The following are typical examples.
\begin{ex}
  $\RCAo$ proves the following.
  For a given effectively coded $\omega$-model $\mathcal W$, the following operators $h_{1},\dots, h_6$ are Turing reducible to $\mathcal W$.
  \begin{enumerate}
  \item For any $e$ and $i$, if the $e$th $\Pi^{W_i}_1$ class is non-empty (in other words, $\forall n \exists \sigma \pi_{0}( n, e, \sigma, W_i\restrict n)$ holds
  \footnote{We adopt this convention that we call a definable set defined by $\Pi^{X}_1$ formula $\Pi^{X}_1$ class.}
  ), $W_{h_1(e,i)}$ is a member of that class.
    \item For any $e$, let $\Phi_e$ be the $e$th $\Pi^{\mathcal W}_1$ sentence. If $h_2(e)= 0$ then $\Phi_e$ is true and if $h_2(e)=1$ then $\Phi_e$ is false.
    \item For given $i$ and $k$, $W_{h_{3}(i,k)}$ is the set by removing the least $k$ elements from $W_{i}$ (here, $W_{h_3(i,k)}=\emptyset$ if $|W_{i}|\le k$).
    \item $h_{4},h_{5},h_{6}$ to describe unions, intersections and complements for sets,
  \end{enumerate}
\end{ex}

For our main constructions, the following realization is the most important one.

\begin{prop}\label{prop:piselect}
Let $R_0(i, X)$ and $R_1(i,X)$ be $\Pi^{X}_{2}$ formulas.
Then the following are provable within $\RCAo$. Let $X$ be a set and $\mathcal W$ is an effectively coded $\omega$-model $\mathcal W = \la W_i \ra$ containg $X'$ (which exists as a set).
Then there exists a function $h\le_{\T}\mathcal W$ such that
  \[
  \forall j [ R_0(j,X) \lor R_1(j,X) \to ( h(j)=0 \to R_0(j,X) ) \land ( h(j)=1 \to R_1(j,X)) ].
  \]
\end{prop}

\begin{proof}
  The essential idea is from the proof of Theorem 6.44 in \cite{slice}.
  For simplicity and our later usage, we consider the case where $R_0$ means that $W \cap W_j$ is infinite and $R_1$ means that $W \cap W_j^c$ is infinite for a given set $W$.
  For the purpose of this paper, we identify $W$ as $W_i$ for some fixed index $i$.

  Define a $X'$-partial computable function $\psi(i,j)$ which finds $n$ and $k$ such that $\forall m \geq n \lnot (W_i(m) = 1 \land W_j(m) = 1-k)$, and outputs $k$.
  From the well-known characterization of PA-degree, $X' \ll Z$ implies that we can define a total $Z$-computable function $\overline \psi$ which extends $\psi$.
  If $\overline \psi(i,j) = 0$ then (A)$W_i \cap W_j$ is finite or (B) $\psi(i,j)$ is undefined.
  In both cases, $W_i \cap W_j^c$ is infinite.
  If $\overline \psi(i,j) = 1$ then (C) $W_i \cap W_j$ is infinite or (B) $\psi(i,j)$ is undefined.
  In both cases, $W_i \cap W_j$ is infinite.
\end{proof}

To simplify our constructions, we use the following propositions.

\begin{prop}\label{prop:dense}
The following is provable within $\RCAo$.
Let $X,Y$ be sets such that $X \ll Y$. Then there exists an effectively coded $\omega$-model $\mathcal W$
  containing $X$ such that $X \ll \mathcal W \ll Y$.
\end{prop}

\begin{prop}\label{prop:rephrase}
  $\ll^2$-basis theorem for $P$ is equivalent to the following statement. For sets $A$ and $C$ such that $A' \ll C$, there exists $B$ which is a $P$-solution to $A$ such that $(A \oplus B)' \ll C$.
\end{prop}

As a lemma, we see the following lemma which gives the full characterization of the $\ll^{n}$-relations.

\begin{lem}\label{lem:ll}
The following is provable within $\RCAo$.
For any sets $X$ and $Y$, $\mathrm{B}\Sigma^{Y}_n$ implies that the following are equivalent.
  \begin{enumerate}
    \item $X \ll^n Y$,
    \item There exists a $\Delta^{Y}_{n}$ set $\mathcal W$ which codes an $\omega$-model of $\WKLo$ containing $X^{(n-1)}$,
    \item Any nonempty $\Pi_1^{X^{(n-1)}}$-class  has a $\Delta^{Y}_{n}$ element.
  \end{enumerate}
\end{lem}

Items 1 and 2 are equivalent even over $\RCAs$; see \cite{iso}, Lemma 4.8 (b).
The equivalence of items 1 and 3 follows immediately from the definitions.

Remark that the assertion ``$\mathcal W$ is an effectively coded $\omega$-model of $\WKLo$ containing $A$'' is $\Pi^{0,A}_1$ over $\WKLo$.

\begin{proof}[\it Proof of \propref{prop:dense}.]
  Since $X \ll Y$, there exists an effectively coded $\omega$-model $\mathcal W_1\le Y$ of $\WKLo$ containing $X$ by \thmref{thm:coded}.
We then obtain another effectively coded $\omega$-model $\mathcal W_2$ containing $X$ in $\mathcal W_1$. Recall the condition {\it 3} in \thmref{thm:coded} that $g_{\mathcal W_2}$ specifies an element of nonempty $\Pi^{\mathcal W_1}_1$-class. It implies $\mathcal W_2 \ll \mathcal W_1$.
  Since $\mathcal W_1 \leq_{\T} Y$, $\mathcal W_2$ is the desired coded model.
\end{proof}

\begin{proof}[\it Proof of \propref{prop:rephrase}.]
  One direction is a simple revision. Note that $X \ll^2 Y$ is equivalent to $X' \ll Y'$.
  For the other direction, assume the $\ll^2$-basis theorem for $P$. Fix sets $A$ and $C$ such that $A' \ll C$.
  By the Jump inversion theorem, we can take a set $D$ whose jump is Turing equivalent to $C$.
  Then we can apply the $\ll^2$-basis theorem for $A$ and $D$ since $A \ll^2 D$.
  It gives us a $P$-solution $B$ such that $(A \oplus B) \ll^2 D$, which implies $(A \oplus B)' \ll D' = C$.
\end{proof}

\section{Low-like basis theorems for the splits of Ramsey's theorem}\label{constructions}

Low$_n$ basis theorems are well studied in the context of Ramsey's theorems.
The $\ll^n$-basis theorem introduced in Section \ref{llrelations} is a generalization of low$_n$-basis theorem.
We emulate some results on low$_n$ basis theorems to discuss $\ll^n$-basis theorems and derive conservation results.
Fiori, Ko\l{}odziejczyk, Wong and Yokoyama essentially discussed the relations between the $\ll^2$-basis theorem and $\Pi^1_1$-conservations in \cite{iso}.

In this section, we prove $\ll^2$-basis theorems for $\COH, \DAN[2], \DAN[\empty]$ and $\sEMinf$.
As for $\COH$ and $\DAN[(2)]$, using the decomposition of $\mathrm{RT}^2_{(2)}$ (\propref{prop:decomposition}) we conclude the provability of the $\ll^2$-basis theorem for $\mathrm{RT}^2_{(2)}$.

\begin{thm}\label{thm:headerforbasistheorem}
  \begin{enumerate}
    \item $\IN[2]$ proves the $\ll^2$-basis theorem for $\RTtt$,
    \item $\IN[2]$ proves the $\ll^2$-basis theorem for $\EMinf$,
    \item $\BIII$ proves the $\ll^2$-basis theorem for $\RTt$.
  \end{enumerate}
\end{thm}

The technique used in the proof og \thmref{thm:headerforbasistheorem}.1 is widely known as the ``first-jump control argument'' in the study of Ramsey's theorem for pairs from the perspective of computability.
The formalization of the proof within $\IN[2]$ (e.g., based on the construction in \cite{slice}) is still straightforward (see also \cite{QuPaYo-wf} for another proof).
Here, we carefully sharpen the construction and obtain the first-jump control proofs for $\EMinf$ and $\RTt$ within appropriate systems.
This provides simpler and unified proofs for some known conservation theorems.
Moreover, the conservation proofs based on the above arguments can be improved to polynomial-time computable proof transformations, which will appear in the forthcoming paper by Ikari-Ko{\l}odziejczyk-Yokoyama.
For the precise statements, see Section \ref{conclusions}.
The rest of the paper is devoted to the proof of \thmref{thm:headerforbasistheorem}.

\subsection{The $\ll^2$-basis theorem for $\COH$}

$\COH$ is an important split of both $\RTtt$ and $\RTt$.
In this subsection we show that the $\ll^2$-basis theorem for $\COH$ is proved from $\IN[2]$ in order to illustrate the main technique, which is a component of the proof of \thmref{thm:headerforbasistheorem}.
Essentially this is a formalization of Theorem 6.44 in \cite{slice}.
More precisely, we prove the following theorem.

\begin{thm}\label{thm:basisforCOH}
  The following is provable within $\IN[2]$.
  For any sets $X$ and $Z$ such that $X' \ll Z'$, and an $X$-computable sequence of sets $\vec R$, there exists an $\vec R$-cohesive set $C$ such that $(X \oplus C)' \ll Z'$.
\end{thm}

\begin{proof}
Let $(M, \cS)$ be a model of $\IN[2]$.
Fix $X$ and a set $Z$ such that $X' \ll Z'$ in ${\cS}$.
Consider the model $(M, \Delta^0_2(M,\cS)) \models \RCAo$, where $\Delta^0_2(M,\cS)$ is the collection of all definable set $Z\subseteq M$ such that $(M, \cS)\models Z\le_{\T} X'$ for some $X\in \cS$.

As a preparation for the proof, let $\effcodf$ and $\effcods$ be effectively coded $\omega$-models of $\WKLo$ such that $X' \in \effcodf, \effcodf \ll Z', \effcods \in \effcodf, X \in \effcods$, and $\effcods' \equiv X'$.
We construct a set $C$ that satisfies the following.
(i) infinite,
(ii) $C \subseteq^\ast R_i$ or $C \subseteq^\ast R_i^c$ for each $i$, and
(iii) the jump of $C$ is Turing reducible to the coded $\omega$-model.
Here we use \thmref{thm:coded}.4 to satisfy $\effcods' = X'$.
We perform Mathias forcing computably in $\effcodf$.
Note that $\effcodf \in \Delta^0_2(M,\cS)$ since $Z' \in \Delta^0_2(M,\cS)$; thus we can use $\INX{1}{\effcodf}$ in this proof.
Here, we use a notion of Mathias forcing that consists of pairs $p = (F,I)$ such that $F \in \effcods$ is a finite set, $I \in \effcods$ is an infinite set, and $\max F < \min I$.
Note that each $F$ is represented by its index of finite set, and each $I$ is represented by its index in the coded model $\effcods$.
We define a partial order $\preceq$ on $\mathbb P$ as follows.
\[
(F,I) \preceq (E, H) :\Leftrightarrow F \supseteq E \land F \setminus E \subseteq H \land I \subseteq H.
\]

To ensure that a generic set satisfies these conditions, we construct a sequence of conditions from $\mathbb P$, and we consider the following requirements for $(F,I) \in \mathbb P$.
Remark that each of them is $\Sigma_1^{I}$ or $\Pi_1^{I}$.

\begin{description}
  \item[$(D_n)$] $I \subseteq R_n \lor I \subseteq R_n^c$,
  \item[$(E_n)$] $|F| \geq n$,
  \item[$(R_e)$] $\Phi_{e,\max F}^{F}(e) \defined$,
  \item[$(N_e)$] $\not\exists D \subseteq_{\text{fin}} I (\Phi_e^{F\cup D}(e) \defined )$.
\end{description}

We construct a sequence of conditions $\la p_s \ra_{s \in \mathbb N}$ stage by stage, satisfying the following.

\begin{equation}\label{eq:cohdesired}
  \forall e \exists s (p_s \text{ satisfies } E_e \land (R_e \lor N_e) \land D_e ).
\end{equation}

At stage $s+1$, pick the least requirement not satisfied yet among $\{ E_n\}_n$,\\
$\{ R_e \lor N_e \}_e$ and $\{ D_n\}_n$.
Note that this selection can be determined within $\effcodf$.

For the requirement $E_n$:
Let $D$ consist of the least $(n - |F_s|)$ elements in $I_s$,
and define $F_{s+1} = F_s \cup D, I_{s+1} = I_s \setminus [0, \max D]$.

For the requirement $R_e \lor N_e$:
Ask $\effcodf$ whether there exists a $D \finsub I_s$ such that $\Phi_{e,\max D}^{F_s \cup D}(e) \defined$.
If the answer is {\sf Yes}, then we fix such a $D \subseteq I_s$ to define $F_{s+1} = F_s \cup D$ and $I_{s+1} = I_s \setminus [0, \max D]$.
If the answer is {\sf No}, then we define $F_{s+1} = F_s$ and $I_{s+1} = I_s$.

For the requirement $D_n$, we define $F_{s+1} = F_s$.
We make $\effcods$ select an infinite set from either $I_s \cap R_n$ or $I_s \cap R_n^c$ and define it to be $I_{s+1}$.

\begin{rem}
  Recall that $\vec R$ is a $X$-computable and $\effcods$ can select an infinite set from either $I \cap R$ or $I \cap R^c$ when $I$ is infinite (\propref{prop:piselect}).
  Additionally, all other operations and judgments appearing above can be done within the effectively coded $\omega$-model $\effcodf$.
\end{rem}

After the construction we define $C = \cup_s F_s$.
We verify this $C$ satisfies the desired conditions.
(i) Infiniteness is clear since for all $n$ there exists a segment $F_s$ of $C$ satisfying $E_n$.
(ii) For $\vec R$-cohesiveness, we fix $e$ and $s$ such that $p_s$ satisfies $D_e$.
From the construction and the definition of $\mathbb P$, every element that would be added into $C$ after the stage $s$ is included in $I_s$. Then $C \setminus R_e \subseteq F_s$ or $C \setminus R_e^c \subseteq F_s$ holds.
This implies $C \subseteq^{\ast} R_e$ or $C \subseteq^{\ast} R_e^c$.
(iii) The verification of $(C \oplus X)' \leq_{\T} \effcodf$ follows a standard argument.
We fix $e$ and take an $s$ such that $p_s$ satisfies $R_e \lor N_e$ and ask $\effcodf$ which one holds.
If $p_s$ satisfies $R_e$ then $\Phi_e^{F_s}(e) \defined$.
This implies $e \in C'$ since $F_s \subseteq C$.
If $p_s$ satisfies $N_e$ then $e \notin C'$ since every element which is added into $C$ after stage $s$ is an element of $I_s$.
This implies $C'$ is computable from $\effcodf$.
Since $\effcodf \ll Z'$, $C$ satisfies desired complexity.
\end{proof}

\begin{rem}
  \begin{enumerate}
    \item Recall that we can use $\INX{1}{\effcodf}$, which helps us to use $\effcods$-primitive recursion (\thmref{thm:primitiverec}).
    Then the above verification process can be carried out successfully.
    \item Proving that the jump of some constructed set is computable from a fixed oracle (an effectively coded $\omega$-model in our cases) is a common technique in recursion theory.
    For discussions in the context of the components of Ramsey's theorems, see in \cite{CJS} or \cite{slice}.
    Indeed we repeat this type of proof in the following sections, where we omit such proofs with only a brief mention.
  \end{enumerate}
\end{rem}

\subsection{The $\ll^2$-basis theorem for Erd\H{o}s-Moser principle}

Recall that the Erd\H{o}s-Moser principle is a meaningful split of Ramsey's theorem.
By a similar discussion to that in the previous section, we prove the $\ll^2$-basis theorem for $\EMinf$.
More precisely, we prove the following theorem.

\begin{thm}\label{thm:basisforEM}
  The following is provable within $\IN[2]$.
  For any number $k$, $k$-coloring $c$ and any set $Z$ such that $c' \ll Z'$,
  there exists an infinite set $B$ such that $c$ is fallow on $B$ and $(c \oplus B)' \ll Z'$.
\end{thm}

\begin{proof}
  From the usual discussion we cen prove $\RCAo \plus \IN[2] \vdash \COH \plus \sEMinf \to \EMinf$.
  Thus, it suffices to prove the $\ll^2$-basis theorem for $\sEMinf$.

  Let $(M, \cS)$ be a model of $\IN[2]$.
  Fix a number $k$, a stable $k$-coloring $c$, and a set $Z$ such that $c' \ll Z'$ in ${\cS}$.
  Consider the model $(M, \Delta^0_2(M,\cS)) \models \RCAo$.

  As preparation for the proof, let $\effcodf$ and $\effcods$ be effectively coded $\omega$-models of $\WKLo$ such that $c' \in \effcodf, \effcodf \ll Z', \effcods \in \effcodf, c \in \effcods$, and $\effcods' \equiv X'$.
  Here we use \thmref{thm:coded}.4 to satisfy $\effcods' = c'$.
  We perform Mathias forcing computably in $\effcodf$.

  We construct a set $B$ such that
  (i) infinite,
  (ii) for any $x,y,z\in B$ such that $x<y<z$, $c(x,z) \in \{ c(x,y), c(y,z)\}$,
  (iii) $B'$ is Turing reducible to the coded $\omega$-model.

  Define $A_i = \{ x \mid \lim_y c(x,y) = i\}$ for each $i<k$.
  We define a notion of Mathias forcing $\mathbb P$ as the set of pairs $p = (F, I)$ satisfying the following.
  (i) $I \in \effcods$ isinfinite,
  (ii) $F \in \effcods$ is finite,
  (iii) $\max F < \min I$,
  (iv) $c$ is fallow on $F$,
  (v) $\forall z \in I$ $c$ is fallow on $F \cup \{ z \}$,
  and (vi) $\forall x \in F, \forall z,z' \in I ( c(x,z) = c(x, z') )$.
  We call an $(F, I)$ {\it precondition} when it satisfies (ii)-(vi),
  and call it {\it condition} when it satisfies (i)-(vi).
  Remark that for a given $(F,I) \in \mathbb P$, $c$ is fallow on $F \cup E$ whenever $E$ is a finite subset of $I$ on which $c$ is fallow. (iv) and (v) guarantee this.

We consider the following statements for a precondition $p = (F,I)$.

\begin{description}
  \item[($E_n^+$)] $|F| \geq n$,
  \item[($R_e$)] $\Phi_{e, \max F}^{F}(e) \defined$,
  \item[($N_e$)] $\not \exists D \finsub I ( c$ is fallow on $F \cup D \land \Phi_e^{F \cup D}(e) \defined )$.
\end{description}

Our requirements for Mathias forcing are $E_n^+, R_e \lor N_e$ for natural numbers $n$ and $e$.

Through the construction, we satisfy these requirements while preserving fallowness.
As terminology for the Mathias forcing construction, we define the following three concepts, {\it positively forcing}, {\it compatibility} and {\it negatively forcing}.

\begin{dfn}
  Let $(F,I)$ be a (pre)condition, $\widetilde J$ be one of $E_n^+$ or $R_e$.
  \begin{enumerate}
    \item $(F,I)$ {\it positively forces} $\widetilde{J}$ when $F$ satisfies $\widetilde{J}$,
    \item $(F,I)$ is {\it compatible} with $\widetilde{J}$ when $\exists E \subseteq I ( F \cup E$  satisfies $\widetilde{J} \land c$ is fallow on $F \cup E)$,
    \item $(F,I)$ {\it negatively forces} $\widetilde{J}$ when $(F,I)$ is not compatible with $\widetilde{J}$.
    In other words, $\forall E \subseteq I$ ($c$ is fallow on $F \cup E \to F \cup E$ does not satisfy $\widetilde{J}$).
    Furthermore, in the case $\widetilde{J} = R_e$ and $(F,I)$ satisfies $N_e$.
  \end{enumerate}
\end{dfn}

\begin{rem}
  \begin{enumerate}
    \item For any condition $(F,I)$ and $n$, $(F,I)$ is compatible with $E_n^+$.
    In other words, we can find an extension to satisfy $E_n^+$ at any stage where $E_n^+$ is focused on.
    \item This categorization for ``satisfaction'' is consistent with the basic strategy of the construction along Mathias forcing.
    Roughly speaking, at each stage we extend a condition $p$ (i) to a condition which positively forces a requirement focused on at the stage if $p$ is compatible with that requirement,
    (ii) to a condition which negatively forces the requirement if $p$ is not compatible with the requirement.
  \end{enumerate}
\end{rem}

At stage $s+1$, for fixed $(F,I) \in \mathbb P$ let $\widetilde{J}$ be one of $E_n^{+}$ or $R_e$ that has not been forced until the stage $s$.
We ask the following $\Sigma^{I}_1$ question which $\effcodf$ can answer.
\begin{equation}\label{eq:EMquestion}
\exists t (\forall \la C_i \ra_{i<k} : \text{ partition of } I \restrict t) (\exists i<k) ((F, C_i) \text{ is compatible with } \widetilde{J}).
\end{equation}
Depending on the answer, we proceed with the construction.
Note that this is $\Sigma^{I}_1$ because we work within an effectively coded $\omega$-model of $\WKLo$ and $\WKLo$ has the compactness.

{\bf Case 1} The answer is {\sf Yes}. Fix such a $t$.
Remark that $\la A_i \cap I \restrict t \ra_{i<k}$ is a $k$-partition of $I \restrict t$.
Then there exists an $i<k$ such that $(F, A_i \cap I\restrict t)$ is compatible with $\widetilde{J}$.
Pick a finite set $E \subseteq A_i \cap I \restrict t$ such that $F \cup E$ satisfies $\widetilde J$ and $c$ is fallow on $F \cup E$.
Define $m$, using $\effcodf$, as the least number such that $\forall x \in F \cup E \forall m' \geq m ( c(x,m) = c(x,m') )$.
Note that such $m$ exists since $c$ is stable.
Moreover, $\effcodf$ decides this value and $E$ is a subset of $A_i$.
Now $(F \cup E, I \setminus [0,m])$ is a suitable extension of $(F,I)$ which positively forces $\widetilde{J}$.

{\bf Case 2} The answer is {\sf No}.
It means $\forall t (\exists \la C_i\ra_{i<k}:$ partition of $I \restrict t ) (\forall i<k) ((F,C_i)$ negatively forces $\widetilde{J})$.
Since $\effcods$ is an effectively coded $\omega$-model of $\WKLo$, it picks a $k$-partition $\la C_i \ra_{i<k}$ of $I$ such that $(\forall i<k) ((F,C_i)$ negatively forces $\widetilde{J})$.
Fix this $\la C_i \ra_{i<k}$.
In the case $\widetilde{J} = R_e$, this implies $\not\exists D \finsub C_i (c$ is fallow on $F \cup D \land \Phi_e^{F \cup D}(e) \defined)$.
Then $N_e$ holds for $(F, C_i)$ for any $i<k$.
We then fix an infinite set $C_i$ to take an appropriate extension $(F, C_i)$.
Note that since $I$ is infinite, at least one $C_i$ must be infinite. We can apply \propref{prop:piselect} (relative to $\effcods$) $k$ times to select an $i$ for which $C_i$ is infinite.

Finally we have a sequence of conditions $\la p_s \ra_s$ such that
\[
\forall e,n \exists s ( (F_s, I_s) \text{ satisfies } E_n^{+} \land (R_e \lor N_e) ).
\]

To conclude \thmref{thm:basisforEM}, we define $B = \cup_s F_s$.
By the construction, it is clear that $c$ is fallow on $B$.
The argument for the infiniteness and for the Turing reducibility of the jump are the same as in the proof of \thmref{thm:basisforCOH}.
\end{proof}

\begin{rem}
  \begin{enumerate}
    \item The basic idea is that we extend finite sets while preserving their fallowness and positive information for $B'$.
    We restrict infinite reservoirs to keep negative information for $B'$ regarding at least their ``fallow candidates''.
    \item The essential idea of the construction is that we consider the finite partitions, making the next condition negatively force a $\Pi^0_1$ requirement if no candidate positively forces the corresponding $\Sigma^0_1$ requirement.
    This is a standard idea in this topic originating from the results by Cholak, Jockusch and Slaman (\cite{CJS}).
    As mentioned in the proof, it is important that we formally realize this method because we work within a coded model of $\WKLo$.
    We aldo use this technique in the next proof.
  \end{enumerate}
\end{rem}

\subsection{The $\ll^2$-basis theorem for $\DAN[2]$ and $\DAN[\empty]$}\label{construction}

As already mentioned, $\DAN[2]$ and $\DAN[\empty]$ are an important split of $\RTtt$ and $\RTt$ respectively.
Moreover, they and the $\ll^2$-basis theorem for them play important roles in our results.
In this subsection, we prove the $\ll^2$-basis theorem for $\DAN[\empty]$ within $\BIII$.
More precisely, we prove the following.

\begin{thm}\label{thm:basisforD2}
  $\BIII$ proves the following statement.
  For any number $k$, set $X$, $\Delta^{X}_2$ $k$-partition $\la A_i \ra_{i<k}$ of $\mathbb N$ and any set $Z$ such that $X' \ll Z'$,
  there exists an infinite set $B$ such that $(X \oplus B)' \ll Z'$ and $B \subseteq A_i$ for some $i$.
\end{thm}

A careful observation of the proof shows that it contains the proof of $\ll^2$-basis theorem for $\DAN[2]$ within $\IN[2]$.

\begin{proof}
Let $(M, \cS)$ be a model of $\BIII$.
Fix $X$, $\Delta^{X}_2$-partition $\la A_i\ra_{i<k}$ of $\mathbb N$ and a set $Z$ such that $X' \ll Z'$ in ${\cS}$.
Consider the model $(M, \Delta^0_2(M,\cS)) \models \RCAo \plus \BN[2]$.

As a preparation for the proof, take two effectively coded $\omega$-models of $\WKLo$ $\effcodf$ and $\effcods$ such that $X' \in \effcodf, \effcodf \ll Z', \effcods \in \effcodf, X \in \effcods$, and $\effcods' \equiv X'$.
Here we use \thmref{thm:coded2} to satisfy $\effcods' = X'$.
We perform Mathias forcing computably in $\effcodf$.
Note that $\effcodf \in \Delta^0_2(M,\cS)$ since $Z' \in \Delta^0_2(M, \cS)$; thus we can use $\mathrm{B}\Sigma^{\effcodf}_2$ in the proof.
We construct a set $B$ such that
(i) infinite,
(ii) for some $i<k$ $B \subseteq A_i$,
and (iii) the jump of $B$ is Turing reducible to the coded $\omega$-model.

We define the notion of Mathias forcing $\mathbb P$ as the sets of tuples $p = (F^0, \dots, F^{k-1}, I)$ satisfying the following.
(i) $F^i \finsub A_i$ for all $i<k$,
(ii) $I \in \effcods$,
(iii) $\max F^i < \min I$ for all $ i<k$n
and (iv) $I$ is infinite.
We call a tuple $p = (F^0, \dots, F^{k-1}, I)$ {\it precondition} when it satisfies (i)-(iii), and call it {\it condition} when it satisfies (i)-(iv).
Recall that $I \in \effcods$ is represented by indices of $\effcods$.
We define a partial order $\preceq$ on $\mathbb P$ as follows.
\[
(F^0, \dots, F^{k-1}, I) \preceq (E^0, \dots, E^{k-1}, H)
:\Leftrightarrow \bigwedge_{i<k} (E^i \subseteq F^i \land F^i \setminus E^i \subseteq H) \land I \subseteq H.
\]

We consider the following statements for a precondition $p = \condition[\empty]$.

\begin{enumerate}
  \item[($\Eposi$)] $|F^i| \geq n$,
  \item[($\Enega$)] $\not \exists D \finsub I (|F^i \cup D| \geq n)$,
  \item[($\posireq$)] $\Phi_{e,\max F^i}^{F^i}(e) \defined$,
  \item[($\negareq$)] $\not \exists D \finsub I \Phi_e^{F^i \cup D}(e) \defined$.
\end{enumerate}

\begin{dfn}
  Let $\{ \enuposi{e}{\empty} \}_e$ be an enumeration of $E_n^i$ and $R_e^i$ defined below.
  \[
  \enuposi{(2e')}{\empty} = \Eposi[e'],
  \enuposi{(2e'+1)}{\empty} = \posireq[e'].
  \]
  We order $\{ \enuposi{e}{\empty} \}_e$ by $\la e, i \ra$.
  For a (pre)condition $p = \condition[\empty]$ and $\enuposi{e}{\empty}$ we define the following terminology.
  \begin{enumerate}
    \item $p$ {\it positively forces} $\enuposi{e}{\empty}$ when $F^i$ satisfies $\enuposi{e}{\empty}$,
    \item $p$ is {\it compatible with} $\enuposi{e}{\empty}$ when $\exists E \finsub I$ $(F^i \cup E, I)$ satisfies $\enuposi{e}{\empty}$,
    \item $p$ {\it negatively forces} $\enuposi{e}{\empty}$ when $p$ is not compatible with $\enuposi{e}{\empty}$.
    In this case $(F^i, I)$ satisfies $\Enega[e/2]$ or $\negareq[(e-1)/2]$,
    \item $\enuposi{e}{\empty}$ is {\it decided by} $p$ when $p$ positively forces $\enuposi{e}{\empty}$ or $p$ negatively forces $\enuposi{e}{\empty}$.
  \end{enumerate}
\end{dfn}

\begin{rem}
  Note that for a given $p=\condition[\empty]$, the statement ``$\enuposi{e}{\empty}$ is decided by $p$'' is $\Pi^{I}_1$, and thus $\effcodf$ can judge whether it holds or not.
\end{rem}

At stage $s+1$, for a given $p_s = \condition[s] \in \mathbb P$,
let $\stageindex{e}{s+1}$ be the least $\stageindex{e}{\empty} < s$ such that $\enuposi{e}{\empty}$ is not decided by $p_{s}$.
We ask the following $\Sigma_1^{I_s}$ question which $\effcodf$ can answer.
\[
  \exists t (\forall \la C_i \ra_{i < k} \text{: partition of } I_s \restrict t )(\exists i <k) ((F_{s}^0, \dots, F_s^{k-1}, I_s \cap C_i) \text{ is compatible with } \enuposi{e}{s+1}).
\]
Depending on the answer, we proceed with the construction.

{\bf Case 1} The answer is {\sf Yes}.
Fix a $t$ such that $(\forall \la C_i \ra_{i < k}$: partition of $I_s \restrict t)$ $(\exists i <k) (F_{s}^0, \dots, F_s^{k-1},$ $I_s \cap C_i)$ is compatible with $\enuposi{e}{s+1}$.
Since $\la A_i \cap I_s \restrict t \ra_{i<k}$ is a $k$-partition of $I_s \restrict t$,
there exists $i<k$ such that $(F_{s}^0, \dots, F_s^{k-1}, A_i \cap I_s\restrict t)$ is compatible with $\enuposi{e}{s+1}$.
Fix a finite set $E \subseteq A_i \cap I_s \restrict t$ which satisfies $\enuposi{e}{s+1}$.
Then the following $p_{s+1}$ is a suitable extension which positively forces $\enuposi{e}{s+1}$ regardless of whether $\stageindex{e}{s+1}$ is even or odd.
\[
F_{s+1}^j=
\begin{cases}
  F_s^j & ( j \neq i)\\
  F_s^i \cup E & ( j = i)
\end{cases},~~
I_{s+1} = I_s \setminus [0, \max E].
\]

{\bf Case 2} The answer is {\sf No}.
With the same discussion in the proof of \thmref{thm:basisforEM},
we can take $\la C_i \ra_{i<k}$ a partition of $I_s$ such that
for all $i<k$ $(F_s^0, \dots, F_s^{k-1}, I_s \cap C_i)$ is not compatible with $\enuposi{e}{s+1}$.
It implies $(F_s^0, \dots, F_s^{k-1}, I_s \cap C_i)$ satisfies $E_{{\stageindex{e}{s+1}}/2}^{i,-}$ or $N_{(\stageindex{e}{s+1})/2}^i$ for each $i<k$.
Pick an $i<k$ such that $I_s \cap C_i$ is infinite and take an appropriate extension $p_{s+1} = (F_s^0, \dots, F_s^{k-1}, I_s \cap C_i)$ which negatively forces $\enuposi{e}{s+1}$.
Note that in this case $\enuposi{e}{s+1}$ must be $N_{(\stageindex{e}{s+1}-1)/2}$ since $I_s \cap C_i$ is infinite.

\begin{rem}
  \begin{enumerate}
    \item In both of Case 1 and Case 2, the following holds for the activated color $i$: $\forall e' \leq \stageindex{e}{s+1}$ $\enuposi{(e')}{\empty}$ is decided by $p_{s+1}$.
    \item From the definition of $\stageindex{e}{s}$'s, at any stage $s$ there exists an $i<k$ such that $\stageindex{e}{s+1} \geq \stageindex{e}{s} + 1$.
    Therefore, $(\sum_{i<k} \stageindex{e}{s}) \leq s$ holds for each $s$.
  \end{enumerate}
\end{rem}

From the construction and the above remarks, the following holds for the constructed sequence $\la p_s \ra_s$.

\begin{equation}\label{eq:prewanted}
  \forall e \exists i<k \exists s \forall e' \leq \stageindex{e}{s} \enuposi{(e')}{\empty} \text{ is decided by } p_s.
\end{equation}

Since the statement ``$\enuposi{(e')}{\empty}$ is decided by $p_s$" is $\Pi^{\effcods}_1$, we can apply $\mathrm{B}\Sigma_2^{\effcodf}$ to (\ref{eq:prewanted}) to get the following.

\begin{equation}\label{eq:appliedBII}
  \exists i<k \forall e \exists s \enuposi{e}{\empty} \text{ is decided by } p_s.
\end{equation}

Fix a color $i$ selected in (\ref{eq:appliedBII}), $p_s$ always positively forces $E_e^{i,+}$ whenever $E_e^{i,+}$ is decided by $p_s$.
This is because $E_e^{i,-}$ is never forced by $p_s$ in case $I_s$ is infinite.

After all, the following holds for the sequence $\la p_s \ra_s$.
\[
\forall e,n \exists s p_s \text{ forces } (R_e^i \lor N_e^i) \land E_n^{i,+}.
\]

Then for the selected color $i$, $B = \cup_sF_s^i$ is a desired set for \thmref{thm:basisforD2}.
From the construction, $B \subseteq A_i$ is clear.
The arguments for the infiniteness and the Turing reducibility of its jump are the same as in the proof of \thmref{thm:basisforCOH}.
\end{proof}

\begin{rem}
  If $k=2$ we do not have to use $\mathrm{B}\Sigma_2^{\effcodf}$ to select a color.
  So the construction and verifications are accomplished within $\mathrm{I}\Sigma_1^{\effcodf}$.
  That is, $\IN[2]$ is strong enough to prove the $\ll^2$-basis theorem for $\DAN[2]$.
\end{rem}

\section{Conclusions}\label{conclusions}

In this paper we constructed $\ll^2$-solutions to $\RTtt$ within $\IN[2]$, $\EMinf$ within $\IN[2]$, and $\RTt$ within $\BIII$ (\thmref{thm:headerforbasistheorem}).
With a slight insight, we find that these results give simpler proofs of the following known conservation results.

\begin{cor}[Towsner and Yokoyama \cite{EM}]
$\WKLo+\EMinf+\RTtt+\IN[2]$  is a $\Pi^{1}_{1}$-conservative extension of $\RCAo+\IN[2]$.
\end{cor}
\begin{cor}[Slaman and Yokoyama \cite{SlYo}]
$\WKLo+\RTt$ is a $\Pi^{1}_{1}$-conservative extension of $\RCAo+\BN[3]$.
\end{cor}

Indeed, we extend these conservation results to the results onproof transformations and proof size.
More specifically, we can show the following type of statement.

\begin{claim}\label{claim:forthcoming1}
  There exists a polynomial-time procedure which for a given $\BIII$-proof of $\RTt$ outputs a $\BIII$-proof of $\RTt$.
\end{claim}

For the proofs we use the method of the forcing interpretations, where we regard basis theorems as completeness in a sense.
As a result, we can also prove the following extended results.
This is a generalization of the result by Simpson, Tanaka and Yamazaki \cite{sty}.

\begin{claim}\label{claim:forthcoming2}
 Let $P$ be a formula of the form $\forall X (\theta(X) \to \exists Y \alpha(X, Y))$ where $\theta$ and $\alpha$ are arithmetical.
  If $\BIII$ proves the $\ll^2$-basis theorem for $P$ then there exists a polynomially $\Pi^1_1$-reflecting forcing interpretation of $\WKLo \plus \BIII \plus P$ in $\BIII$.
\end{claim}

Combining \claimref{claim:forthcoming1} with the results of this paper, we obtain \claimref{claim:forthcoming2}.
We will show the details in a forthcoming paper with Ko\l{}odziejczyk.

\section*{Acknowledgements}
Ikari was partialy supported by JST, the establishment of university fellowships towards the creation of science technology innovation, Grant Number JPMJFS2102.
Yokoyama is partially supported by
 JSPS KAKENHI grant numbers JP19K03601, JP21KK0045 and JP23K03193.

  \bibliographystyle{plain}
  \bibliography{bibRT_working.bib}
\end{document}